\documentclass[12pt]{amsart}
\usepackage{amsfonts}
\usepackage{amsmath}
\usepackage{amsthm}
\usepackage{amssymb}
\usepackage{latexsym}
\usepackage{epsfig}
\usepackage{url,titletoc}
\usepackage{multicol}
\makeindex

\usepackage{tabularx}
\usepackage{enumerate,mathrsfs}

\usepackage[pagebackref,colorlinks,linkcolor=red,citecolor=blue,urlcolor=blue,hypertexnames=true]{hyperref}

\usepackage{titletoc}

\newcommand{\csim}{\stackrel{\mathrm{cyc}}{\thicksim}}

\newcommand{\NCSOStools}{\href{http://ncsostools.fis.unm.si/}{{\tt NCSOStools}} }
\newcommand{\NCSOStoolz}{\href{http://ncsostools.fis.unm.si/}{{\tt NCSOStools}}}

\newcommand{\NCAlgebra}{\href{http://www.math.ucsd.edu/~ncalg/}{{\tt NCAlgebra}} }

\DeclareMathOperator{\Tr}{Tr}
\def\bes{\begin{equation*} }
\def\ees{\end{equation*} }

\def\sym{\operatorname{Sym}}

\def\beq{\begin{equation}}
\def\eeq{\end{equation}}

\def\ben{\begin{enumerate}}
\def\een{\end{enumerate}}

\newtheorem{theorem}{Theorem}[section]
\newtheorem{lemma}[theorem]{Lemma}

\newtheorem{thm}[theorem]{Theorem}
\newtheorem{prop}[theorem]{Proposition}

\theoremstyle{definition}

\newtheorem{exa}[theorem]{Example}

\newtheorem{defn}[theorem]{Definition}

\makeatletter
\newcommand{\mycontentsbox}{%
{\centerline{NOT FOR PUBLICATION}
\small\tableofcontents}}
\def\enddoc@text{\ifx\@empty\@translators \else\@settranslators\fi
\ifx\@empty\addresses \else\@setaddresses\fi
\newpage\mycontentsbox}
\makeatother

\def\R{\mathbb{R}}
\def\RR{\mathbb{R}}

\def\ax{\langle x \rangle}

\def\N{\mathbb{N}}

\def\bo{\mathbf{1}}

\numberwithin{equation}{section}

\begin{document}

\setcounter{tocdepth}{3}
\contentsmargin{2.55em} 
\dottedcontents{section}[3.8em]{}{2.3em}{.4pc} 
\dottedcontents{subsection}[6.1em]{}{3.2em}{.4pc}
\dottedcontents{subsubsection}[8.4em]{}{4.1em}{.4pc}

\title{On trace-convex noncommutative polynomials}
\author[Igor Klep]{Igor Klep${}^{1}$}
\address{Igor Klep, Department of Mathematics, 
The University of Auckland, New Zealand}
\email{igor.klep@auckland.ac.nz}
\thanks{${}^1$Supported by the Marsden Fund Council of the Royal Society of New Zealand. Partially supported by the Slovenian Research Agency grants P1-0222, L1-4292 and L1-6722. Part of this research was done while the author was on leave from the University of Maribor.}

\author[Scott McCullough]{Scott A. McCullough${}^2$}
\address{Scott McCullough, Department of Mathematics\\
  University of Florida\\ Gainesville %\\
   % Box 118105\\
   %  Gainesville, FL 32611-8105\\
   %  USA
   }
   \email{sam@math.ufl.edu}
\thanks{${}^2$Research supported by the NSF grants DMS 1101137 and 1361501}

\author[Christopher Nelson]{Christopher S. Nelson${}^3$}
\address{Christopher Nelson, Department of Mathematics\\
  University of California \\
  San Diego}
\email{csnelson@math.ucsd.edu}
\thanks{${}^3$Partly supported by the National Science Foundation, DMS 1201498}



\dedicatory{Dedicated to Bill Helton on the occasion of his ?? birthday}

\begin{abstract}
To each continuous function $f:\R\to\R$ there is an associated trace function on $n\times n$ real symmetric
matrices $\Tr f$. The classical Klein lemma states that $f$ is convex if and only if $\Tr f$ is 
convex. In this note we present an algebraic strengthening of this lemma for univariate polynomials $f$: 
$\Tr f$ is convex if and only if the noncommutative second directional derivative of $f$ is a sum of 
hermitian squares and commutators in the free algebra. We also give a localized version of this result.
\end{abstract}

%%%%%%%%%% text-only abstract
\iffalse
To each real continuous function f there is an associated trace function on real symmetric matrices Tr f. The classical Klein lemma states that f is convex if and only if Tr f is convex. In this note we present an algebraic strengthening of this lemma for univariate polynomials f: Tr f is convex if and only if the noncommutative second directional derivative of f is a sum of hermitian squares and commutators in the free algebra. We also give a localized version of this result.
%%%%%%%%%%%
\fi

\maketitle

\def\bS{\mathbb S}

\section{Introduction}

Trace-convexity is a notion frequently used in free probability
and free analysis \cite{SV06,KV}, where e.g.~the 
trace of certain potentials is  assumed to be convex;
see \cite{Gui,Car,GS} or the references therein.
One of the basic technical tools of the trade is the so-called Klein lemma saying that 
a continuous  function $f:\R\to\R$ is convex if and only if the associated trace
function $\Tr f: \bS_n \to \R$ is convex for all $n\in\N$.
(Here $\bS_n$ denotes the set of all real symmetric $n\times n$ matrices.)
We call such a function {\bf trace-convex}.

In this note we  establish an algebraic version of Klein's lemma.
That is, we give an algebraic certificate using sums of squares and commutators
in the free algebra on two letters whose existence is equivalent to trace-convexity of a polynomial.
 Indeed, we show that trace-convexity of a univariate real polynomial $p$ is equivalent to its second noncommutative
  derivative $p^{\prime\prime}(x)[h]$ being a sum of hermitian  squares plus commutators. 

The article is organized as follows. Section \ref{sec:prelim} fixes notation, terminology, and
gives some preliminaries. Then Section \ref{sec:main} contains our main results, 
and we conclude with remarks and algorithmic considerations in Section \ref{sec:finito}.

\section{Notation and preliminaries}\label{sec:prelim}

\subsection{Matrices}
  There is a natural partial ordering on $\bS_n$
  defined by $X\succeq Y$ if the symmetric
  matrix $X-Y$ is positive semidefinite; i.e., if its
  eigenvalues are all nonnegative.  
  Similarly, $X\succ Y$, if $X-Y$ is positive definite; 
  i.e., all its eigenvalues are positive.

  \subsection{Noncommutative $($nc$)$ polynomials}
  Even if $p$ is a univariate polynomial, it naturally has noncommutative derivatives which
  are polynomials in two freely noncommuting variables. 

Let $x=(x_{1},\ldots,x_{g})$ denote a $g$-tuple
 of free noncommuting variables
 and let  $\RR\ax$ denote the associative
$\RR$-algebra freely generated by $x$. Its elements are called {\bf $($nc$)$ polynomials}.
An element of the form $aw$ where $0\neq a\in \RR$ and
$w$ is a {\bf word} in the variables $x$
  is called a {\bf monomial} and $a$ its
{\bf coefficient}.
The empty word $\varnothing$ is the multiplicative identity for $\RR\ax$.

 There is a natural {\bf involution} ${}^T$  on $\RR\ax$ 
 that reverses words. For example, 
$$(2-  3 x_{1}^2 x_{2} x_{3})^T =2  -3 x_{3} x_{2} x_{1}^2.$$
  A polynomial $p$ is a {\bf symmetric polynomial} 
if $p^T=p$.  
  Because $x_j^{T}=x_j$ we refer to the variables 
  as {\bf symmetric variables}.
 The {\bf degree}  of an nc polynomial $p$, denoted $\deg(p)$, is 
 the length of the longest word appearing in $p$. 
 Let $\RR\ax_k$ \index{$\RR\ax_k$} denote the 
 polynomials of degree at most $k$.

\subsection{Derivatives}
Given a polynomial $p\in\R\ax$, 
the $\ell^{\rm th}$ {\bf noncommutative directional derivative} of $p$ 
 in the {\it direction} $h$ is
\bes
 p^{(\ell)}(x)[h]:=\left.\frac{d^\ell p(x+ t h)}{d t^\ell}\right|_{t=0}.
\ees
 Thus $p^{(\ell)}(x)[h]$ is the polynomial 
that evaluates to 
$$
\left.
\frac{d^\ell p(X+tH)}{dt^\ell}\right|_{t=0}\quad\textrm{for every $n\in\N$, and every choice of tuples }
X,H \in \bS_n^g.
$$
 Let 
  $p^{\prime}(x)[h]$ denote the noncommutative first derivative of $p$ and we denote 
 the {\bf Hessian}, the second noncommutative derivative of $p$ in the direction $h$, by
 $p^{\prime\prime}(x)[h]$. 
Equivalently, the Hessian of $p$ can also be defined as the part of the noncommutative
polynomial
$$r(x)[h]:=2\big(p(x+h)-p(x)\big)\in \RR\ax[h]:=\RR \langle x_1,\dots,x_g,\, h_1,\dots,h_g\rangle$$
that is homogeneous of degree two in $h$.

If $p^{\prime\prime} \neq 0$, that  is, if $p$ is an nc polynomial of degree two or more, 
then its Hessian $p^{\prime\prime}(x)[h]$ is a polynomial in 
the $2g$ variables $x_1,\ldots,x_g,h_1\ldots,h_g$
which is  homogeneous of degree two in $h$,  and
has degree equal to the degree of $p$.

\subsection{Commutators, cyclic equivalence}
A polynomial of the form $[p,q]:=pq-qp$ for $p,q\in\R\ax$ is a {\bf commutator}.
Two polynomials
$f,g\in \R\ax$ are called {\bf cyclically equivalent} ($f\csim g$)
if $f-g$ is a sum of commutators in $\R\ax$.
Cyclic equivalence can be easily checked, cf.~\cite[Remark 1.3]{KS}.

\subsection{Trace-convexity}
\def\cD{\mathcal D}

We now introduce the central notion used in this article.

\begin{defn}
A symmetric polynomial $p\in\R\ax$ is {\bf trace-convex} 
  if for each $n$ and each pair of $g$ tuples of $n\times n$
  symmetric matrices $X=(X_1,\dots,X_g)$ and $Y=(Y_1,\dots,Y_g)$,
 we have
\beq
\label{eq:1}
    \frac12 \big(\Tr p(X)+ \Tr p(Y)\big) \geq \Tr p\Big(\frac{X+Y}{2}\Big). 
\eeq
  Equivalently,
\begin{equation}
\label{eq:2}
    \frac{\Tr p(X)+ \Tr p(Y)}2 - \Tr p\Big(\frac{X+Y}{2}\Big)  \geq 0.
\end{equation}
We sometimes restrict \eqref{eq:1} or, equivalently, \eqref{eq:2}, to hold only for
$X,Y$ in a domain $\cD$. In this case $p$ is {\bf trace-convex on $\cD$}.
\end{defn}

\subsection{Related notions}
 Noncommutative polynomials and noncommutative rational functions arise
 in several contexts including systems theory \cite{BGMa, BGMb} and
 is a part of the new field of free (freely noncommutative) analysis \cite{SV06,KV,AM}.  Sums of squares
 representations are a theme in both the commutative and noncommutative settings \cite{Putinar}.

There is a related notion of {\bf matrix-convexity} of an nc polynomial. 
A symmetric polynomial $p\in\R\ax$ is matrix-convex if 
\[  \frac{p(X)+p(Y)}2 - p\Big(\frac{X+Y}{2}\Big)  \succeq 0\]
for all tuples of symmetric matrices $X,Y$.  Note, that even if $p$
 is a univariate polynomial, $p(x+y)$ is an nc polynomial in $x,y$ 
 as $X$ and $Y$ need not commute.
This convexity condition is  very strong  and leads to extreme rigidity.
For instance, by Helton-McCullough \cite{HM}, every matrix-convex polynomial is of degree at most two, 
 a result which depends on  two  key observations.  First,  $p$ is matrix-convex if and only if  $p^{\prime\prime}(x)[h]$ is matrix-positive,
i.e., 
\[
p^{\prime\prime}(X)[H]\succeq0
\]
for all $X,H\in \bS_n^g$ and $n\in\N$ \cite{CHSY}. Second, $q\in\R\ax$ is matrix-positive
if and only if it is a sum of hermitian squares, that is,
\[
q= \sum_j r_j^T r_j\]
for some $r_j\in\R\ax$ \cite{Hel, McC}.

We refer the reader to \cite{Kra,Eff,Han,OST,Uch}  for further studies of operator-monotonicity.

\section{Results}\label{sec:main}
We are now ready to present the  main results of this article
characterizing univariate trace-convex polynomials with  algebraic
certificates involving  sums of  squares.

\begin{theorem}[global version]
\label{thm:global}
If $p$ is a univariate polynomial, then the following are equivalent.
\ben[\rm(i)]
\item \label{it:i}
$p$ is convex;  % $($as a commutative polynomial$)$;
\item\label{it:ii}
$\operatorname{Tr} p$ is convex, i.e., $p$ is trace-convex;
\item\label{it:iii}
$p''(x)[h]$ is a sum of hermitian squares and commutators.
\een
\end{theorem}

Next we present a local version of Theorem \ref{thm:global}
characterizing univariate  polynomials which are trace-convex on a matrix-interval.
This time the  algebraic
certificates involve  weighted sums of  squares.

\begin{thm}[local version]\label{thm:local}
Suppose $p$ is a univariate polynomial and  $-\infty < a<b < \infty$.
\ben[\rm(1)]
\item
 $\operatorname{Tr} p$ is convex on $aI \prec X \prec bI$ if and
only if 
$p$ is convex on $(a,b)$ if and only if
$p''(x)[h]$ is cyclically equivalent to a polynomial of the form
\begin{align}
\label{eq:local}
% p''(x)[h] &\csim
 \sum_i^{\rm finite} q_i(x)[h]^T q_i(x)[h]&+
 \sum_{i}^{\rm finite} r_i(x)[h]^T(x-a)r_i(x)[h]\\
 \notag
 &+\sum_i^{\rm finite} s_i(x)[h]^T (b-x) s_i(x)[h]+
 \sum_{i}^{\rm finite} t_i(x)[h]^T(x-a)(b-x)t_i(x)[h]
 \\
 \notag
 &+ \sum_{i}^{\rm finite} (x-a)u_i(x)[h]^T(b-x)u_i(x)[h]
\end{align}
for some $q_i, r_i, s_i, t_i, u_i\in\R\ax[h]$ homogeneous of degree one in $h$.
 \item 
 $\Tr p$ is convex on $bI \prec X$ if and only if
 $p$ is convex on $(b, \infty)$ if and only if
$p''(x)[h]$ is cyclically equivalent to a polynomial of the form
\[\sum_i^{\rm finite} q_i(x)[h]^T q_i(x)[h] + \sum_{j}^{\rm finite} r_j(x)[h]^T(x-b)r_j(x)[h].\]
 \item 
  $\Tr p$ is convex on $ X \prec aI$ if and only if
$p$ is convex on $(-\infty, a)$ if and only if
$p''(x)[h]$ is cyclically equivalent to a polynomial of the form
\[\sum_i^{\rm finite} q_i(x)[h]^T q_i(x)[h] + \sum_{j}^{\rm finite}
r_j(x)[h]^T(a-x)r_j(x)[h].\]
\een
\end{thm}

The proofs of Theorems \ref{thm:global} and \ref{thm:local} occupy the remainder
 of this section.

\subsection{Symmetrizer}

One of our main tools in analyzing second derivatives of
univariate nc polynomials is the following operation of symmetrization.

\begin{defn}
Given $d\in\N$ and nc polynomials $y_1,\dots,y_d$, define $\sym_d(y_1, \ldots, y_d)$ to be the
nc polynomial
 \[
 \sym_d(y_1, \ldots, y_d) = \frac{1}{d!} \sum_{\sigma \in S_d}
y_{\sigma(1)} \cdots y_{\sigma(d)}.
 \]
 Often will omit the subscript $d$.
\end{defn}

\begin{lemma}
Let $y_1, \ldots, y_d$ be nc polynomials.
If $\sigma \in S_d,$ then
\[
 \sym_d(y_{\sigma(1)}, \ldots, y_{\sigma(d)}) = \sym_d(y_1, \ldots, y_d).
\] 
\end{lemma}

\begin{proof}
Trivial. 
\end{proof}

\begin{lemma}\label{lem:symmDist}
  Fix $1\le k\le d$. 
   If $y_1, \ldots, y_d$ be nc polynomials and $a_k$ is a constant, then
%\begin{multline*}
\[
\begin{split}
  \sym_d(y_1, \ldots, y_{k-1}, y_k - a_k, y_{k+1}, \ldots, y_d) = &
  \sym_{d}(y_1, \ldots, y_{k-1}, y_k, y_{k+1}, \ldots, y_d) \\
& -
a_k\sym_{d-1}(y_1, \ldots, y_{k-1}, y_{k+1}, \ldots, y_d).
\end{split}
\]
%\end{multline*}
\end{lemma}

\begin{proof}
The polynomial $\sym_d(y_1, \ldots, y_{k-1}, y_k - a_k, y_{k+1},
\ldots, y_d)$ is a sum of products of polynomials, each of which contains
 $y_k - a_k$. Distributing the $y_k -
a_k$ in each product gives
\begin{align}
 \label{eq:sumSymm}
\frac{1}{d!} \sum_{\sigma \in S_d} y_{\sigma(1)} \cdots y_{\sigma(d)}
- a_k \frac{1}{d!} \sum_{j=1}^d \sum_{\sigma(j) = k} y_{\sigma(1)} \cdots
y_{\sigma(j-1)}y_{\sigma(j+1)}\cdots y_{\sigma(d)}.
\end{align}
 For $1\le j\le d$, let $S_{d,j}$ denote those $\sigma\in S_d$ such
 that $\sigma(j)=k$.  
Given a permutation $\sigma \in S_{d,j}$, define 
$\tilde{\sigma}$ to be
\[
 \tilde{\sigma}(i) = \left\{
 \begin{array}{cc}
  k& i=1\\
  \sigma(i-1)& 1 \leq i-1 < j\\
  \sigma(i)& j < i
 \end{array}
\right.
\]
Given a fixed $j$, the mapping from $S_{d,j}$ to $S_{d,1}$
 defined by  $\sigma \mapsto \tilde{\sigma}$  
  is a bijection and further, 
\[
 y_{\tilde{\sigma}(2)} \cdots y_{\tilde{\sigma}(d)} = y_{\sigma(1)} \cdots
    y_{\sigma(j-1)} y_{\sigma(j+1)} \cdots y_{\sigma(d)}.
\]
Therefore (\ref{eq:sumSymm}) simplifies to
\begin{multline*}
 \frac{1}{d!} \sum_{\sigma \in S_d} y_{\sigma(1)} \cdots y_{\sigma(d)}
- a_k \frac{1}{(d-1)!} \sum_{\sigma(1) = k} y_{\sigma(2)} \cdots
\cdots y_{\sigma(d)} \\
 =     \sym_{d}(y_1, \ldots, y_{k-1}, y_k, y_{k+1}, \ldots, y_d)-
a_k\sym_{d-1}(y_1, \ldots, y_{k-1}, y_{k+1}, \ldots, y_d).
\qedhere
\end{multline*}
\end{proof}

\subsection{From commutative to noncommutative polynomials}
\def\ds{\displaystyle}
 For a univariate polynomial $p$, 
 let $\ds \frac{dp}{dx}$ and $\ds \frac{d^2 p}{dx^2}$ denote the ordinary
   first and second derivative of $p$ (and $p^{\prime}$ and $p^{\prime\prime}$ the first and second nc derivative of $p$).

\begin{lemma}
\label{lem:psymm}
Let $p$ be a univariate polynomial.
If 
\[\frac{d^2p}{dx^2} = (x - a_1) \ldots (x - a_d),\]
 then %$p^{\prime\prime}(x)[h]$ equals
\[
 p^{\prime\prime}(x)[h] = \sym_{d+2}(x-a_1, \ldots, x-a_d, h, h).
\]
\end{lemma}

\begin{proof}First consider the simple case  $\ds\frac{d^2p}{dx^2}(x) = x^d$
 for which
\[p = \frac{1}{(d+2)(d+1)} x^{d+2} + \ell,\]
where $\ell$ is some linear
polynomial. Computing the second nc
derivative of $x^{d+2}$ gives $2$ times the sum of all words of degree $d$ in $x$ and
degree two in $h$.
Therefore, for each $i < j$, the coefficient of the word with an $h$ as the
$i^{\rm th}$
and $j^{\rm th}$ letters in $p''(x)[h]$ is
\[\frac{2}{(d+2)(d+1)} .\]
Examining $\sym(x, \ldots, x, h, h)$, we see,
for each $i < j$, the coefficient of the word with an $h$ as the $i^{\rm th}$
and $j^{\rm th}$ letters is
\[\frac{1}{(d+2)!} d!2! = \frac{2}{(d+2)(d+1)}.\]

Next consider the general case $\ds\frac{d^2p}{dx^2}= (x - a_1) \ldots (x - a_d)$.  We
see that
\[\frac{d^2 p}{dx^2} = \sum_{k=0}^{d} \left( \sum_{\substack{1 \leq i_1 < \cdots < i_{d-k}
\leq n}}
(-1)^{d-k} a_{i_1}\ldots
a_{i_{d-k}} x^k\right).\]
By linearity,
\[
 p^{\prime\prime}(x)[h] = \sum_{k=0}^{d} \sum_{1 \leq i_1 < \cdots < i_{d-k} \leq n}
(-1)^{d-k} a_{i_1}\ldots
a_{i_{d-k}} \sym_{k+2}(x, \ldots, x, h, h).
\]
Repeated application of Lemma \ref{lem:symmDist} to each $(x-a_i)$ in the
expression $\sym(x-a_1, \ldots, x-a_d, h, h)$ shows that
 this last expression is equal to $\sym(x-a_1, \ldots, x-a_d, h, h)$.
\end{proof}

\subsection{On Hankel matrices}

Recall that an $(n+1)\times (n+1)$ square matrix $T$ is called {\bf Hankel}
if it has  constant anti-diagonals, i.e.,
\[
T_{i,j}=T_{i-1,j+1}
\]
for $0<i\leq n$ and $0\leq j<n$. A finite sequence $c_0,\ldots,c_{2m}$ generates
an $(m+1)\times( m+1)$ Hankel matrix $T$ with 
\[
T_{i,j}= c_{i+j}, \qquad 0\leq i,j\leq m.
\]
We refer to \cite{Dym} for more on Hankel matrices.

\begin{lemma}
\label{lem:middPSD}For each $k, \ell, d \in \mathbb{N}$ with $\ell\le k$,   there is a regular
 Borel measure supported on $\mathbb R$ such that
\[
   \frac{1}{{{2d+k}\choose{j+\ell}}} = \int x^j \, d\mu,
\]
 for $0\le j\le 2d.$  In particular, the Hankel matrix generated
by the finite sequence
%$\ds\frac{1}{{{2d+k}\choose{\ell}}},\ldots ,\frac{1}{{{2d+k}\choose{2d+\ell}}}$
$\ds \left( \frac1{{{2d+k}\choose{\ell+j}} }\right)_{j=0}^{2d}$
 is positive semidefinite. 
\end{lemma}

\begin{proof}
Consider the measure
\[ 
  d\mu = \frac{2d+k+1}{(x+1)^{2d+k+2}} x^k \, dx
\] 
  on the half line $[0, \infty)$. The $j^{th}$ moment of this
measure is
\[ \int_0^{\infty} x^{j+k} \frac{2d+k+1}{(x+1)^{2d+k+2}}\, dx. \]  Using the
substitution $x = \frac{1}{t} - 1$ gives
\[ \int_0^{1} (2d+k+1)(1-t)^{j+k} t^{2d+k - (j+k)} \, dt = (2d+k+1)\ B\big(2d+k-(j+k)+1, (j+k)+1\big) =
\frac{1}{{{2d+k}\choose{j+k}}},\]
where $B(x,y)$ denotes the beta function.
\end{proof}

\begin{lemma}
 \label{lem:bo}
 Given a positive integer $d$ and $n_0,n_1,\dots,n_d\in\mathbb N$, let $\bo_{i,j}$ 
  denote the $n_i\times n_j$ matrix all of whose entries are $1$.  If $H = (H_{i,j})_{i,j=0}^d$
 is a positive semidefinite $(d+1)\times (d+1)$ matrix, then the block matrix
 (of total size $n\times n$ where $n=\sum n_j$)
\[
  C = \begin{pmatrix}   H_{i,j} \bo_{i,j} \end{pmatrix}_{i,j=0}^d
\]
 is also positive semidefinite.
\end{lemma}

\begin{proof}
  Let $N$ denote the max of $\{n_0,\dots,n_d\}$ and let $\bo$ denote the $N\times N$ matrix
  each of whose entries is $1$.  The tensor (Kronecker) product $H\otimes \bo$ is positive
  semidefinite since both $H$ and $\bo$ are.  Finally, $C$ is obtained from $H\otimes \bo$
  by compressing to a subspace (removing appropriate rows and columns) and is thus positive semidefinite.
\end{proof}

\subsection{A uniform representation of the symmetrizer}

\begin{prop}
 \label{lem:bsPSDMulti}
%Let $b_1, \ldots, b_{2d}$ be scalars.
 Fix a positive integer $d$ and nonnegative integer $k$.  
There exists a vector-valued polynomial $W(x,h,c)$ in the nc variables $x$ and $h$
 and commuting variables $c=(c_1,\dots,c_d)$ (thus each $c_j$ commutes with all other variables)
 and positive semidefinite matrices $C$
 and $C_0,\dots,C_k$  such that
% for each $d$-tuple of scalars
%$c$ the entries of $W(x,h,c)$ are nc polynomials in $x$ and $h$, and such that
%the following hold true:
\begin{enumerate}[\rm (1)]
 \item the polynomial $\sym(x-b_1,\ldots, x-b_{2d}, h_1, h_2)$ (in the nc variables $x,h_1,h_2$
 and commuting variables $b_1,\dots,b_{2d}$) is cyclically
equivalent to
\begin{align}
\label{eq:simple}
 W(x,h_1,b_1, \ldots, b_d)^T \,C\, W(x, h_2,b_{d+1}, \ldots, b_{2d});
\end{align}
%where $C$ is a constant positive semidefinite matrix.
 \item %Let $a_1, \ldots, a_k$ be scalars.  
  the polynomial
     $\sym(x-a_1, \ldots, x-a_k, x-b_1, \ldots,x-b_{2d}, h_1, h_2)$ 
  (in the nc variables $x,h_1,h_2$ and commuting variables $a_1,\dots,a_k;b_1,\dots,b_{2d}$) is cyclically
equivalent to
\begin{align}
\label{eq:fullCycSum}
\sum_{\ell=0}^{k} \sum_{\tau \in S_k} &(x-a_{\tau(1)}) \cdots (x-a_{\tau(\ell)})
W(x,h_1,b_1, \ldots, b_d)^T (x-a_{\tau(\ell+1)}) \cdots\\
\notag
& \cdots (x-a_{\tau(k)}) C_{\ell}
W(x,h_2,b_{d+1}, \ldots, b_{2d}).
\end{align}
\end{enumerate}
\end{prop}

\begin{proof}
 Note that Equation \eqref{eq:simple} follows from Equation \eqref{eq:fullCycSum} by choosing $k=0$.
The polynomial $\sym(x-a_1, \ldots, x-a_k,x-b_1,\ldots,x-b_{2d}, h_1, h_2)$ is a sum of products which can be cyclically permuted so
that they are of the form
\begin{align}
\label{eq:cyclB}
(x-a_{\tau(1)}) &\cdots (x-a_{\tau(\ell)}) (x - b_{\sigma(1)}) \cdots (x -
b_{\sigma(m)})\ h_1\ (x-a_{\tau(\ell+1)}) \\
\notag
&\cdots (x-a_{\tau(k)}) (x -
b_{\sigma(m+1)})
\ldots (x -
b_{\sigma(2d)})\ h_2,
\end{align}
for some $0 \leq \ell \leq k$, $0 \leq m \leq 2d$, $\sigma \in S_{2d}$, $\tau
\in S_{k}$.
There are $2d + k +2$ cyclic permutations of (\ref{eq:cyclB}) within
$\sym(x-a_1, \ldots, x-a_k,x-b_1,
\ldots,
x-b_{2d}, h_1, h_2)$.
Further, the factors of the form $x - c$, with $c$ a scalar, can be commuted
with each other as long as
they don't pass over an $h_j$. Given a set of factors which
appear to the left of $h_1$, the other side of $h_1$ must consist of the
remaining unused factors. This means that up to cyclic permutation
each term is uniquely determined by which factors appear between $h_1$ and
$h_2$.  Therefore the polynomial
$\sym(x-a_1, \ldots, x-a_k,x-b_1,\ldots,x-b_{2d}, h_1, h_2)$,
up to cyclic equivalence,
contains a product of the form (\ref{eq:cyclB}) with coefficient
equal to
\begin{equation}
 \label{eq:coefficient}
 \frac{(\ell + m)!(2d +k - \ell - m)!(2d+k+2)}{(2d+k+2)!} =
\frac{1}{{{2d+k}\choose{\ell + m}}(2d+k+1)}.
\end{equation}

By cyclic permutation, we can express (\ref{eq:cyclB}) uniquely in the form
\begin{equation}
  \label{eq:groupBterms}
  (x - a_{\tau(1)}) \cdots (x - a_{\tau(\ell)})
  (x - b_{\phi(1)}) \cdots (x - b_{\phi(d-r)})
  h_1
  (x - b_{\phi(d-r+1)}) \cdots (x - b_{\phi(d)})
\end{equation}
\[
  (x - a_{\tau(\ell+1)})
   \cdots (x - a_{\tau(d)})
  (x - b_{d + \rho(1)}) \cdots (x - b_{d + \rho(s)})
  h_2
  (x - b_{d + \rho(s+1)}) \cdots (x - b_{d + \rho(d)})
\]
where $\phi, \rho \in S_d$.  Let $f_s(x, h, c_1, \ldots, c_d)$ be a
vector consisting of all unique polynomials in the nc variables
$x, h$ and
commuting variables $c = (c_1, \ldots, c_d)$ of the form
\[
  (x - c_{\rho(1)})\cdots(x-c_{\rho(s)})h(x - c_{\rho(s+1)}) \cdots
  (x - c_{\rho(d)})
\]
where $\rho \in S_d$.
Note that the length of $f_s$ is ${{d}\choose{s}}$.
Also, let $\bo_{r, s}$ denote the ${{d}\choose{r}} \times
{{d}\choose{s}}$ matrix all of whose entries are $1$.
Then for each $\ell, r, s$,
\begin{equation}
  \label{eq:FrFsSum}
  \frac{1}{\ell!(k-\ell)!} \sum_{\tau \in S_k}
  (x - a_{\tau(1)}) \cdots (x - a_{\tau(\ell)})
  f_r(x,h_1,b_1, \ldots, b_d)^T
\end{equation}
\[
  (x - a_{\tau(\ell+1)})
  \cdots (x- a_{\tau(d)})
  \bo_{r, s}
  f_s(x,h_2,b_{d+1}, \ldots, b_{2d})
\]
is equivalent to the sum of all distinct polynomials of the form
(\ref{eq:groupBterms}), each of which has coefficient
\begin{equation}
  \label{eq:coeffFrFsSum}
  \frac{1}{{{2d+k}\choose{\ell+2d-(r+s)}} (2d+k+1)}
\end{equation}
in $\sym(x-a_1, \ldots, x-a_k, x-b_1, \ldots, x-b_{2d}, h_1,
h_2)$.

Let
\[
  W(x, h, c_1, \ldots, c_d) = \begin{pmatrix}
    f_0(x, h, c_1, \ldots, c_d)
    \\ \vdots \\
    f_d(x, h, c_1, \ldots, c_d)
  \end{pmatrix},
\]
let
\[
  H_{\ell} = \frac{1}{(2d+k+1)\ell!(k-\ell)!}
  \begin{pmatrix}
    \frac{1}{{{2d+k}\choose{\ell}}} &
    \frac{1}{{{2d+k}\choose{\ell + 1}}} & \cdots
    & \frac{1}{{{2d+k}\choose{\ell + d}}}\\
    \frac{1}{{{2d+k}\choose{\ell + 1}}} &
    \frac{1}{{{2d+k}\choose{\ell + 2}}} & \cdots
    & \frac{1}{{{2d+k}\choose{\ell + d + 1}}}\\
    \vdots & \vdots & \cdots & \vdots \\
    \frac{1}{{{2d+k}\choose{\ell + d}}} &
    \frac{1}{{{2d+k}\choose{\ell + d + 1}}} & \cdots
    & \frac{1}{{{2d+k}\choose{\ell + 2d}}}
  \end{pmatrix},
\]
and define
\[
  C_{\ell} = ((H_{\ell})_{r,s} \bo_{r, s}).
\]
By Lemma \ref{lem:middPSD}, $H_{\ell}$ is positive semidefinite
and thus, by Lemma \ref{lem:bo}, $C_{\ell}$ is positive
semidefinite.
Further, when we sum the polynomials (\ref{eq:FrFsSum}) multiplied by
coefficients (\ref{eq:coeffFrFsSum}) over all $(\ell, r, s)$, we get (\ref{eq:fullCycSum}).
\end{proof}

\subsection{Connecting trace-convexity and trace-positivity}
A set $\ds \cD\subseteq \bigcup_n \bS_n^g=:\bS^g$
is called open if each $\cD(n):=\cD\cap\bS_n^g$ is open.

\begin{lemma}\label{lem:convex}
A symmetric nc polynomial $p\in\R\ax$ is trace-convex
on an open set $\ds\cD\subseteq \bS^g$
if and only if its Hessian $p^{\prime\prime}(x)[h]$ is trace-positive
on $\cD\times\bS^g$, 
i.e.,
\beq\label{eq:tr-pos}
\Tr p''(X)[H]\geq0
\eeq
for all $n\in\N$, $X\in\cD(n)$ and $H\in \bS_n^g$. 
\end{lemma}  

\begin{proof}
  First suppose that $p$ is trace-convex on $\cD$.  
  In this case, given $n\in\N$, $X\in\cD(n)$ and
  $H\in \bS_n^g$, the polynomial of $t\in\mathbb R$,
\[
 q(t) = \Tr \big( p(X+tH) + p(X-tH) -2p(X)\big) 
\]
 takes nonnegative values (for small enough $t$) 
 and $q(0)=0$.  Hence, 
\[
 0\le \frac{d^2 q}{dt^2}(0) = \Tr p^{\prime\prime}(X)[H].
\]
 Thus $p^{\prime\prime}$ is trace-positive on $\cD\times\bS^g$.

 Now suppose $p^{\prime\prime}$ is trace-positive 
on $\cD\times\bS^g$ 
 and let $n\in\N$, $X\in\cD(n)$ and
  $H\in\bS_n^g$ be given such that $X\pm H\in\cD$. Consider the real polynomial 
\[
  r(t) = \Tr \big(p(X+tH)+p(X-tH)-2p(X)\big).
\]
  Observe that $r(0)=0,$ 
\[
  \frac{dr}{dt}(0) = \Tr\big( p^{\prime}(X)[H]-p^{\prime}(X)[H]\big) =0
\]
 and
\[
  \frac{d^2 r}{dt^2}(c) = \Tr\big( p^{\prime\prime}(X+cH)[H] +  p^{\prime\prime}(X-cH)[H] \big).
\]
  In particular, by hypothesis the second (ordinary) derivative of $r$ is  nonnegative.
 By  Taylor's  Theorem, there is a $1>c>0$ such that
\[
  \Tr\big( p(X+H)+p(X-H)-2p(X) \big)=r(1)=\frac12 \frac{d^2 r}{dt^2}(c) \ge 0
\]
 and thus  $p$ is trace-convex on $\cD$.
\end{proof}  

The parallel between trace-convexity and matrix-convexity stops here due
to the failure of a tracial version of Helton's sum of squares theorem \cite{Hel,McC}. That is,
a trace-positive nc polynomial is not necessarily a sum of hermitian squares and 
commutators \cite{KS}. 
For more on 
matrix-positive polynomials see e.g.~\cite{Hel, HM',PNA}, and for
trace-positive nc polynomials we refer to \cite{BK,BCKP} 
(see also \cite{CDT}) and the
references therein.

\subsection{Proof of the main results}

\begin{proof}[Proof of Theorem {\rm\ref{thm:global}}]
Suppose \eqref{it:i} holds. Then $\ds\frac{d^2 p}{dx^2}$ is
nonnegative on $\R$, which implies that it is a sum of polynomials of the form $A (x -
b_1)^2 \ldots (x - b_d)^2$ with $A \geq 0$. Lemma \ref{lem:psymm} and Proposition
\ref{lem:bsPSDMulti} now imply that $p''(x)[h]$ is a sum of squares plus commutators, i.e.,
\eqref{it:iii} holds.

If \eqref{it:iii} holds, that is, $p''(x)[h]$ is a sum of hermitian squares plus commutators,
then it is clear that $\operatorname{Tr}(p''(x)[h])$ is nonnegative (that is, $p^{\prime\prime}$ is trace positive), which
implies, by Lemma \ref{lem:convex}, that $\operatorname{Tr}(p(x))$ is convex, establishing \eqref{it:ii}. Finally,
that \eqref{it:ii} implies \eqref{it:i} is obvious.
\end{proof}

\begin{proof}[Proof of Theorem {\rm\ref{thm:local}}]
(1) If $p''(x)[h]$ is of the form \eqref{eq:local}, then it is clear that for $aI
\prec X \prec bI$, and all $H$, 
$\Tr p''(X)[H]$ is nonnegative.
(Note that $(X-aI)(bI-X)\succ 0$ since $X-aI$ and $bI-X$ commute.)
So, by Lemma \ref{lem:convex},  $\Tr p$ is convex on $aI\prec X\prec bI$ and $p$ is
convex on $(a,b)$.

Conversely, if $p$ is convex on $(a,b)$,
then $\ds\frac{d^2 p}{dx^2}(x)$ is nonnegative for $x \in (a,b)$, which implies that it is a
sum of polynomials each of one of the following forms \cite{PR}:
\begin{enumerate}[\rm(a)]
 \item $A(x - b_1)^2 \cdots (x-b_d)^2$
 \item $A(x-a)(x-b_1)^2 \cdots (x-b_d)^2$
 \item $A(b-x)(x-b_1)^2 \cdots (x-b_d)^2$
 \item $A(x-a)(b-x)(x-b_1)^2 \cdots (x-b_d)^2$
\end{enumerate}
where $A > 0$ and $d \geq 0$.
The result now follows from applying Lemma \ref{lem:psymm} and Proposition
\ref{lem:bsPSDMulti}.

The proofs of (2) and (3) are similar and left as an exercise for the reader.
\end{proof}

\section{Concluding remarks}\label{sec:finito}

\subsection{An alternative proof of the equivalence between trace-convexity and convexity}
Here we present an alternative proof (cf.~\cite[p.~74]{Gui09}) of the equivalence between trace-convexity and
convexity for univariate nc polynomials which avoids the sum of squares certificates. 
 The argument applies naturally to continuous functions $p:(a,b)\to\mathbb R$.

\begin{prop}Let $a, b \in [-\infty, \infty]$ with $a < b$.
If $p:(a,b) \to\mathbb R$ is continuous, then 
$\operatorname{Tr} p$ is convex
on $aI \prec X \prec bI$ if and only if % the commutative collapse $\ch p(x)$ of
$p$ is convex on $(a,b)$.
\end{prop}

\begin{proof}
If $\operatorname{Tr} p$ is convex on $(a,b),$ then for $1\times 1$ matrices
between $aI$ and $bI$,
the trace $\operatorname{Tr} p$ is equal to $p$ and hence $p$
is convex on $(a,b)$.

For the converse,
fix $t \in [0, 1]$ and $X, Y \in \bS_n$ for some
 $n\in\N$ such that $aI \prec X,Y \prec bI$. Then $t X + (1-t)
Y$ is a symmetric matrix and can be decomposed as $O^T \Lambda O$, where $O$ is
orthogonal and $\Lambda$ is diagonal
with entries ${\lambda}_1, \ldots, {\lambda}_n$. Further, $$(O^T \Lambda O)^n =
O^T {\Lambda}^n O \quad\text{ and }\quad \operatorname{Tr}O^T {\Lambda}^n O =
\operatorname{Tr}O O^T {\Lambda}^n = \operatorname{Tr} {\Lambda}^n.$$
Extending this to general polynomials,  we have
\[
\operatorname{Tr}p(O^T \Lambda O) = \operatorname{Tr}p(\Lambda) =
\sum_{i=1}^{n}  p({\lambda}_i).
\]

Conjugating by $O$ gives \[\Lambda = t O X O^T + (1-t) O Y O^T.\]
Since both $X$ and $Y$ are symmetric, they can be decomposed as $X = O^T_X
{\Lambda}_X O_X$ and $Y = O^T_Y {\Lambda}_Y O_Y$,
where $O_X$ and $O_Y$ are orthogonal and $\Lambda_X$ and $\Lambda_Y$ are
diagonal with entries
$\lambda^X_i$ and $\lambda_i^Y$, respectively.

Let $O^T O_X = A = (a_{ij})_{1 \leq i,j \leq n}$ and $O^T O_Y= B =
(b_{ij})_{1 \leq i,j \leq n}$, which are both orthogonal matrices.  The
$ii^{th}$ entry of the matrix $t\ O x O^T = t\ A^T {\Lambda}_X A$ is 
\[
(t\ A^T {\Lambda}_x A)_{ii} = \sum_{k=1}^{n} t\ a^2_{ik}{\lambda}^X_k. 
\]
 Similarly, the
$ii^{th}$ entry of $(1-t)\ B^T {\Lambda}_Y B$ is 
\[
((1-t)\ B^T {\Lambda}_y)_{ii} = \sum_{k=1}^{n} (1-t)\ b^2_{ik}{\lambda}^Y_k.
\]
Adding the two together gives
\[
{\lambda}_i = \sum_{k=1}^{n} t\ a^2_{ik}{\lambda}^X_k + \sum_{k=1}^{n} (1-t)\
b^2_{ik}{\lambda}^Y_k.
\]

Since $A$ and $B$ are orthogonal, $\displaystyle\sum_{k=1}^{n}
a^2_{ik} = \displaystyle\sum_{k=1}^{n} b^2_{ik} = 1$. Therefore
${\lambda}_i$ is a convex combination of the ${\lambda}^X_i$ and
${\lambda}^Y_i$ terms.  
Further, since $aI\prec X,Y\prec bI$, each of the $\lambda_i^X$ and
$\lambda_i^Y$ is in $(a,b)$.
Therefore, if $ p$ is convex on $(a,b)$, then 
\begin{align}
 \notag
  \operatorname{Tr} \big(tX +(1-t)Y\big)&=  \operatorname{Tr}p\big(t O X O^T + (1-t) O
Y O^T\big)\\
\notag
&= \sum_{i=1}^{n}
 p\left(\sum_{k=1}^{n} t a^2_{ik}{\lambda}^X_k + \sum_{k=1}^{n} (1-t)
b^2_{ik}{\lambda}^Y_k\right)\\
\notag
&\leq \sum_{i=1}^{n} \left(\sum_{k=1}^{n} t a^2_{ik}
p({\lambda}^X_k) +
\sum_{k=1}^{n} (1-t)b^2_{ik} p({\lambda}^Y_k)\right) \\
\notag
&=
\sum_{k=1}^{n}\left(t \Big(\sum_{i=1}^{n} a^2_{ik}\Big)  p({\lambda}^X_k) +
\Big(\sum_{i=1}^{n} b^2_{ik}\Big) p({\lambda}^Y_k)\right)\\
\notag
&= t \displaystyle\sum_{k=1}^{n}  p({\lambda}^X_k) + (1-t)
\displaystyle\sum_{k=1}^{n}  p({\lambda}^Y_k)\\
\notag
&= t \operatorname{Tr}p(X) +
(1-t) \operatorname{Tr}p(Y).
\end{align}
Hence  $\operatorname{Tr}p(X)$ is convex on $aI \prec X \prec bI$.
\end{proof}

\subsection{Algorithmic aspects}
There are now several computer algebra packages  available
capable of assisting work in free convexity and free real algebraic geometry. Namely

\ben[\rm(1)]
\item
\NCAlgebra \cite{HOSM} running under Mathematica;
\item 
\NCSOStools \cite{CKP} running under MATLAB.
\een
The former is more universal in that it implements
manipulation with noncommutative variables, including
nc rationals, and several algorithms pertaining to 
convexity. The latter is focused on free positivity,
sums of squares
and numerics. 

\begin{exa}
Here is a simple example computed with the aid of \NCSOStoolz.
We demonstrate our results on $p=15x^2 - 5x^4 + x^6$.
Note that $\ds\frac{d^2p}{dx^2}=(x-1)^2 (x+1)^2$, so $p$ is convex.
To compute a noncommutative trace-convexity certificate we proceed
as follows:
\begin{verbatim}
>> NCvars x;
>> p=15*x^2 - 5*x^4 + x^6;
>> [iscConvex,g,sohs,s] = NCisCycConvex(p,10e-10);
>> iscConvex
iscConvex =
     1
>> sohs
sohs = 
   5.47722558*h1-1.82574651*h1*x^2-1.82573252*x*h1*x-1.82574651*x^2*h1
                                     0.007151308*h1*x-0.007145828*x*h1
              -1e-09*h1*x^2+0.000279894*x*h1+2e-09*x*h1*x-1e-09*x^2*h1
                1.63298798*h1*x^2-0.204119565*x*h1*x-1.42870107*x^2*h1
                                 0.790214359*x*h1*x-0.790862166*x^2*h1
                                                    0.000187011*x^2*h1 
\end{verbatim}

As {\tt iscConvex}${}=1$, we believe $p$ is trace-convex.
To obtain an exact (symbolic) proof as opposed to the numerical
evidence presented above, we proceed as follows.
We try to manually find a sum of squares and commutators certificate of $p''(x)[h]$ as \NCSOStools then outputs more 
intermediate results
which we can analyze.

\begin{verbatim}
>> p2=NC2d(p);
>> [IsCycEq,X,base,sohs,g,SDP_data,L] = NCcycSos(p2);
>> X
X =
   30.0000   -0.0000    0.0000  -10.0000   -9.9999  -10.0000
   -0.0000    0.0001   -0.0001    0.0000    0.0000   -0.0000
    0.0000   -0.0001    0.0001   -0.0000    0.0000    0.0000
  -10.0000    0.0000   -0.0000    6.0000    3.0000    1.0003
   -9.9999    0.0000    0.0000    3.0000    3.9994    3.0000
  -10.0000   -0.0000    0.0000    1.0003    3.0000    6.0000
>> base
base = 
    'h1'
    'h1*x'
    'x*h1'
    'h1*x*x'
    'x*h1*x'
    'x*x*h1'
\end{verbatim}

Here {\tt X} is the Gram matrix corresponding to our sum of squares SDP (semidefinite program), and {\tt base} the corresponding border vector. 
The second and third row and column of {\tt X} are a direct
summand corresponding to a polynomial cyclically equivalent to 0.
Thus with
\[
X=\begin{pmatrix}
30 & -10 & -10 & -10 \\
-10 & 6 & 3 & 1\\
-10 & 3 & 4 & 3\\
-10 & 1 &3 & 6
\end{pmatrix},
\qquad
v=\begin{pmatrix}
h \\ hx^2 \\ xhx \\ x^2h
\end{pmatrix}.
\]
it is easy to verify 
\[
v^TXv \csim p''
\]
and that $X\succeq0$.
Factor $X = R^T R$, where, with 
$\ds
s=\frac1{\sqrt 6},\; t=\sqrt{\frac52}, \; u=\frac{\sqrt{30}}3,
$
\[
R=\begin{pmatrix}
0&0&0&0\\
0&s&-2s&s\\
0&-t&0&t\\
-3u&u&u&u
\end{pmatrix}.
\]
Letting $u=Rv$, the three nonzero entries of $u$,
\[
\begin{split}
q_1 &= s(hx^2-2xhx +x^2h) \\
q_2&=t(hx^2-x^2h) \\
q_3 &=u(3h-hx^2-xhx-x^2h)
\end{split}
\]
satisfy
\[
p''(x)[h] \csim q_1^*q_1+q_2^*q_2+q_3^*q_3.
\]
\end{exa}

\linespread{1.1} %setup so refs+authors end on page


\begin{thebibliography}{KCBP13}


\bibitem[AM+]{AM}
J. Agler, J.E. McCarthy:
Pick Interpolation for free holomorphic functions,
to appear in {\em Amer. J. Math.},
\url{https://arxiv.org/abs/1308.3730}

\bibitem[BGM05]{BGMb}
   J.A. Ball, G.  Groenewald, T. Malakorn:
   Structured noncommutative multidimensional linear systems,
  {\em SIAM J. Control Optim.} {\bf 44} (2005) 1474--1528.
  
\bibitem[BGM06]{BGMa}
   J.A. Ball, G.  Groenewald, T. Malakorn:
Conservative structured noncommutative multidimensional linear systems. 
  The state space method generalizations and applications, 179--223, {\em Oper. Theory Adv. Appl.}, 
  {\bf 161} Birkhäuser, Basel, 2006. 


\bibitem[BCKP13]{BCKP}
S. Burgdorf, K. Cafuta, I. Klep, J. Povh:
The tracial moment problem and trace-optimization of polynomials, 
{\em Math. Program.} {\bf 137} (2013)  557--578.


\bibitem[BK12]{BK}
S. Burgdorf, I. Klep:
The truncated tracial moment problem,
{\em J. Operator Theory} 
{\bf 68} (2012)  141--163. 



\bibitem[CKP11]{CKP}
K. Cafuta, I. Klep, J. Povh:
\href{http://ncsostools.fis.unm.si}{\tt NCSOStools}: a computer algebra system for symbolic and numerical computation with noncommutative polynomials, 
{\em Optim. Methods Softw.} {\bf 26} (2011) 363--380.
 Available from \url{http://ncsostools.fis.unm.si}




 \bibitem[CHSY03]{CHSY} 
J.F. Camino, J.W. Helton, R.E. Skelton, J. Ye:
Matrix inequalities: A symbolic procedure to determine
 convexity automatically, {\em Integral Equations  Operator Theory}
{\bf 46} (2003) 399--454.

\bibitem[Car10]{Car}
E. Carlen:
Trace inequalities and quantum entropy: an introductory course,
 {\em Entropy and the quantum}  73--140, 
Contemp. Math. {\bf 529}, Amer. Math. Soc.,  2010.

\bibitem[CDT10]{CDT} B.  Collins, K.J. Dykema, F. Torres-Ayala:
Sum-of-squares results for polynomials related to the
Bessis-Moussa-Villani conjecture. {\em J. Stat. Phys.} {\bf 139} (2010)
779--799.

\bibitem[Dym07]{Dym}
H. Dym:
{\em Linear algebra in action},
 Graduate Studies in Mathematics {\bf 78},
  Amer. Math. Soc., 2007.


\bibitem[Eff09]{Eff}
E.G. Effros:
  A matrix convexity approach to some celebrated quantum inequalities,
{\em   Proc. Natl. Acad. Sci. USA} {\bf 106} (2009) 1006--1008.

\bibitem[Gui06]{Gui}
A. Guionnet:
 Random matrices and enumeration of maps,
{\em  International Congress of Mathematicians} {\bf III},  623--636, Eur. Math. Soc., 2006.

\bibitem[Gui09]{Gui09}
A. Guionnet: 
{\em Large random matrices}, Lecture Notes in Mathematics {\bf 1957}, Springer, 2009.

\bibitem[GS09]{GS}
A. Guionnet, D. Shlyakhtenko:
Free diffusions and matrix models with strictly convex interaction,
{\em Geom. Funct. Anal.} {\bf  18}  (2009) 1875--1916.

\bibitem[Han97]{Han}
F. Hansen: Operator convex functions of several variables,
{\it Publ. Res. Inst. Math. Sci.} {\bf 33}  (1997) 443--463.

\bibitem[Hel02]{Hel}
J.W. Helton: ``Positive''
noncommutative polynomials are sums of squares,
{\em Ann. of Math. (2)} {\bf 156} (2002) 675--694.


\bibitem[HM04a]{HM} 
 J.W. Helton, S. McCullough:
 Convex noncommutative polynomials have degree two or less, 
 {\em 
SIAM J. Matrix Anal. Appl.} 
  {\bf 25} (2004) 1124--1139.
  
  \bibitem[HM04b]{HM'}
J.W. Helton, S. McCullough:
A Positivstellensatz for noncommutative polynomials,
{\em Trans. Amer. Math. Soc.} {\bf 356} (2004) 3721--3737.

\bibitem[HOSM+]{HOSM}
J.W. Helton, M.C. de Oliveira, M. Stankus, R.L. Miller:
\href{http://math.ucsd.edu/~ncalg}{\tt NCAlgebra}, 2013 release edition. Available from 
\url{http://math.ucsd.edu/~ncalg}

\bibitem[KV+]{KV} D.S. Kaliuzhnyi-Verbovetskyi, V. Vinnikov:
 {\em Foundations of Free Noncommutative Function Theory,}
 to appear in Math. Surveys and Monographs, AMS,
\url{https://arxiv.org/abs/1212.6345}

\bibitem[KS08]{KS}
I. Klep, M. Schweighofer:
 Connes' embedding conjecture and sums of Hermitian squares,
 {\em Adv. Math.} {\bf 217} (2008) 1816--1837.

\bibitem[Kra36]{Kra}
F. Kraus: 
\"Uber konvexe Matrixfunktionen,
{\it Math. Z.} {\bf 41}  (1936) 18--42.
 
\bibitem[McC01]{McC}
 S. McCullough, Factorization of operator-valued polynomials in several non-commuting variables,
 {\it Linear Algebra Appl.} {\bf 326} (2001) 193--203.

\bibitem[OST07]{OST}
H. Osaka, S. Silvestrov, J. Tomiyama:
Monotone operator functions, gaps and power moment problem,
{\it Math. Scand.} {\bf 100}  (2007) 161--183.

\bibitem[PNA10]{PNA}
S. Pironio, M. Navascues,  A. Acin:
Convergent relaxations of polynomial optimization problems with non-commuting variables,
{\em SIAM J. Optimization} {\bf 20} (2010) 2157--2180.

\bibitem[PR00]{PR}
V. Powers, B. Reznick:
Polynomials that are positive on an interval,
{\em Trans. Amer. Math. Soc.} {\bf  352}  (2000) 4677--4692.


\bibitem[Pu13]{Putinar} M. Putinar: 
 Sums of Hermitian squares: old and new,
 in: {\em  Semidefinite
optimization and convex algebraic geometry} edited by
G. Blekherman, P.A. Parrilo, R. Thomas,
407--469, MOS-SIAM Ser.
Optim. {\bf 13}, SIAM, 2013.

\bibitem[SV06]{SV06}
D. Shlyakhtenko, D.-V. Voiculescu:
Free analysis workshop summary: American institute of mathematics,
\url{http://www.aimath.org/pastworkshops/freeanalysis.html}


\bibitem[Uci02]{Uch}
M. Uchiyama: 
Operator monotone functions and operator inequalities,
{\it Sugaku Expositions} {\bf 18}  (2005) 39--52.



\end{thebibliography}
\end{document}